\documentclass[12pt]{article}

\usepackage{amssymb}
\usepackage{amsmath}
\usepackage{amsthm}
\usepackage{bm}
\usepackage{hyperref}
\usepackage{thmtools}
\usepackage{thm-restate}

\newtheorem{proposition}{Proposition}

\begin{document}

\title{\Large A geometric proof of the Brenti--Welker identity}
\author{Ognjen Papaz\thanks{Faculty of Philosophy, University of East Sarajevo, Bosnia and Herzegovina {\tt ognjen.papaz@ff.ues.rs.ba}}}
\date{}
\maketitle

\begin{abstract}
We construct a hypersimplicial subdivision of the $r$-dilation of the $i$-th hypersimplex of dimension $d$ that provides a geometric proof of the Brenti--Welker identity.
\end{abstract}

\section{Introduction and preliminaries}

In \cite{Brenti} Brenti and Welker studied, for a rational formal power series of the form
\[f(t):=\sum_{n\geq 0}a_nt^n=\frac{h(t)}{1-t}^d,\ a_n\in\mathbb C,\]
the transformation of the numerator polynomial when passing for some number $r\geq 1$ to the generating function $f^{\langle r\rangle}(t):=\sum_{n\geq 0}a_{rn}t^n=\frac{h^{\langle r\rangle}(t)}{(1-t)^d}$. One of the main results they have obtained is an identity involving Eulerian numbers.  

An Eulerian number $A(d,i)$ counts the number of permutations $\pi$ of the set $\{1,2,\ldots,d\}$ that have $i-1$ descents, the positions where $\pi(t)>\pi(t+1)$.

Denote with $\mathcal C(r,d,i)$ the set of weak compositions of $i$ into $d$ parts where no part is greater then $r$, i.e
\[\mathcal C(r,d,i)=\{\bm{v}\in\mathbb Z^d_+:v_1+v_2+\cdots+v_d=i, v_t\leq r, 1\leq t\leq d\}.\]
Let $C(r,d,i)=|\mathcal C(r,d,i)|$.

\begin{restatable}{proposition}{myMainProp}\cite[Prop. 2.3]{Brenti}{(Brenti--Welker identity)}\label{identity} Let $d,r\geq 1$. Then
\[\sum_{j=1}^dC(r-1,d+1,ir-j)A(d,j)=r^dA(d,i).\]
\end{restatable}

A famous result from Laplace \cite{Laplace} states that the normalized volume of the $i$-th hypersimplex of dimension $d$ is the Eulerian number $A(d,i)$. 

Recall that $i$-th hypersimplex of dimension $d$, denoted with $\Delta_{d,i}$, is the convex hull of the characteristic vectors of the subsets of $[d+1]$ that have $i$ elements. In short
\[\Delta_{d,i}=\mathrm{conv}\left\{\bm e_T:T\in\binom{[d+1]}{i}\right\}.\]
We have that 
\[\Delta_{d,i}=\{\bm{x}\in\mathbb R^{d+1}:x_1+x_2+\cdots+x_{d+1}=i, 0\leq x_t\leq 1,\ \text{for}\ t\in[d+1]\}.\]

The first triangulation of the hypersimplex that showed that the normalized volume of $\Delta_{d,i}$ is $A(d,i)$ was constructed by Stanley \cite{Stanley}. This triangulation is referred to as the alcoved triangulation. Standard references for the alcoved triangulations of polytopes are \cite{Lam} and \cite{Lam2}.

In \cite{Valencia} Valencia demonstrated a combinatorial proof of Proposition \ref{identity}  by considering the alcoved triangulation of the $r$-dilation of the hypersimplex $\Delta_{d,i}$. He established a bijective correspondence between the sets 
\[[r]^d\times \mathcal U(d,i)\ \text{and}\ \bigcup_{j=1}^d\mathcal C(r-1,d+1,ir-j)\times\mathcal U(d,j)\]
and the set of alcoves in that triangulation. (Here $\mathcal U(d,j)$ denotes the set of permutations of $[d]$ with $j-1$ descents.)
 
We take a more direct approach, avoiding the alcoved triangulation, by constructing a hypersimplicial decomposition of $r$-dilation of $\Delta_{d,i}$ that is parametrized with $\bigcup_{j\in[d]}\mathcal C(r-1,d+1,ir-j)$. More precisely, we prove the following theorem.

\begin{restatable}{theorem}{myMainThm}\label{subdivision} The family of (translates of) hypersimplices
\[\mathcal H_{r,d,i}=\bigcup_{j\in[d]}\{\bm v+\Delta_{d,j}:\bm v\in\mathcal C_{r-1,d+1,ir-j}\}.\]
is a subdivision of $r\Delta_{d,i}$, the $r$-dilation of the hypersimplex $\Delta_{d,i}$.
\end{restatable}

A polytopal subdivision of a polytope $P$ is a collection of non-overlapping polytopes whose union is $P$.

Proposition \ref{identity} follows directly from Theorem \ref{subdivision} because the volume of $r\Delta_{d,i}$ is equal $r^dA(d,i)$ and the sum of the volumes of hypersimplices from $\mathcal H_{r,d,i}$ is equal $\sum_{j=1}^dC(r-1,d+1,ir-j)A(d,j)$.

\section{Results}

We consider the lattice translates of $d$-dimensional hypersimplices in $\mathbb R^{d+1}$, i.e. the translates of the form $\bm v+\Delta_{d,j}$ where $\bm v\in \mathbb Z^{d+1}$ and  $j\in[d]$. 

Note that each of these translates is contained in some hyperplane $x_1+x_2+\cdots+x_{d+1}=i$ where $i\in\mathbb Z$.

Let $\bm x\in\mathbb R^{d+1}$ and $x_1+x_2+\cdots+x_{d+1}\in\mathbb Z$. It is not hard to see that $\bm x$ is contained in $\bm v+\Delta_{d,j}\ (\bm v\in\mathbb Z^{d+1},j\in[d])$ if and only if the following hold:

i) if $\{x_t\}>0$, then $v_t=\lfloor x_t\rfloor$; (Here $\{x_t\}$ and $\lfloor x_t\rfloor$ are fractional and integer part of $x_t$ respectively.)

ii) $\#\{t\in [d+1]:x_t=v_t+1\}+\#\{t\in[d+1]:\{x_t\}>0\}=j$.

Denote with $O(\bm x)$ the set of indices of coordinates of $\bm x$ that have positive fractional part, i.e. $O(\bm x)=\{t\in[d+1]:\{x_t\}>0\}$. Let $o(\bm x)=\sum_{t\in O(\bm x)}\{x_t\}$. Note that $o(\bm x)\in\mathbb Z_+$ and $1\leq o(\bm x)\leq |O(\bm x)|-1$ if $O(\bm x)\neq\emptyset$.

The subfamily of the lattice translates of $d$-dimensinal hypersimplices that contain $\bm x$ is given with 
\begin{equation*}\label{containment}
\{\lfloor \bm x \rfloor-\bm e_T+\Delta_{d,|T|+o(\bm x)}:T\subseteq [d+1]\setminus O(\bm x),|T|+o(\bm x)\in [d]\};
\end{equation*}
here $\lfloor \bm x \rfloor=(\lfloor x_1 \rfloor,\lfloor x_2 \rfloor,\ldots,\lfloor x_{d+1} \rfloor)$.

Let us now consider the intersection of two translates $\bm u+\Delta_{d,j_1}$ and $\bm v+\Delta_{d,j_2}$, where $\bm u,\bm v\in\mathbb Z^{d+1}$ and $j_1,j_2\in [d]$. Suppose that their intersection is nonempty. Then $|u_t-v_t|\leq 1$ for every $t\in[d+1]$ and $j_1+\sum_{t\in [d+1]}u_t=j_2+\sum_{t\in[d+1]}v_t$ . Denote $X_{\bm u}=\{t\in [d+1]:v_t=u_t+1\}$ and $X_{\bm v}=\{t\in[d+1]:u_t=v_t+1\}$. Note that $|X_{\bm u}|\leq j_1$ and $|X_{\bm v}|\leq j_2$ and $j_1-|X_{\bm u}|=j_2-|X_{\bm v}|$. One can readily check that 
\begin{multline}\label{intersections}
\bm u+\Delta_{d,j_1}\cap \bm v+\Delta_{d,j_2}=\\
\bm u+\bm e_{X_{\bm u}}+\mathrm{conv}\left\{\bm e_T:T\in\binom{[d+1]\setminus(X_{\bm u}\cup X_{\bm v})}{j_1-|X_{\bm u}|}\right\}=\\
\bm v+\bm e_{X_{\bm v}}+\mathrm{conv}\left\{\bm e_T:T\in\binom{[d+1]\setminus(X_{\bm u}\cup X_{\bm v})}{j_2-|X_{\bm v}|}\right\}
\end{multline}

From here we can see that the intersection of the hypersimplices $\bm u+\Delta_{d,j_1}$ and $\bm v+\Delta_{d,j_2}$ is a face of both of them. 

Based on the previous considerations, the following proposition holds.

\begin{proposition}\label{triangulation} The family of (translates of) hypersimplices
\[\{\bm v+\Delta_{d,j}:\bm v\in \mathbb Z^{d+1}, v_1+v_2+\cdots+v_{d+1}=i-j, j\in [d]\}\]
generates (hyper) triangulation of the hyperplane $x_1+x_2+\cdots+x_{d+1}=i$.
\end{proposition}

\myMainThm*

\begin{proof}
Because of Proposition \ref{triangulation}, it is enough to prove that $r\Delta_{d,i}=\bigcup\mathcal H_{r,d,i}$. Obviously $\bm v+\Delta_{d,j}\subseteq r\Delta_{d,i}$ for each $\bm v\in \mathcal C_{r-1,d+1,ir-j}$ and each $j\in[d]$. Let $\bm x\in r\Delta_{d,i}$. Suppose that $O(\bm x)=\emptyset$. (Recall that $O(\bm x)=\{t\in[d+1]:\{x_t\}>0$.) Pick any subset $S$ of $[d+1]$ such that $\emptyset\neq S\neq[d+1]$ and $x_t>0$ for $t\in S$. Then, $|S|\in [d]$ and $\bm x-\bm e_S\in \mathcal C_{r-1,d+1,ir-|S|}$. Hence, $\bm x\in\bigcup\mathcal H_{r,d,i}$. Let $O(\bm x)\neq\emptyset$. Now pick any subset $P$ of $[d+1]\setminus O(\bm x)$ such that $x_t>0$ for $t\in P$. Then, $|P|+o(\bm x)\in[d]$, because $1\leq o(\bm x)\leq |O(\bm x)|-1$, and  $\lfloor \bm x \rfloor-\bm e_P\in \mathcal C_{r-1,d+1,ir-|P|-o(\bm x)}$. Again, we can conclude that $x\in\bigcup\mathcal H_{r,d,i}$.
\end{proof}

\myMainProp*

\begin{proof} The number $r^dA(d,i)$ is equal to the normalized volume of $r\Delta_{d,i}$. By Theorem \ref{subdivision}, the normalized volume of $r\Delta_{d,i}$ is equal to $\sum_{j=1}^dC(r-1,d+1,ir-j)A(d,j)$.
\end{proof}

From Proposition \ref{triangulation} it also follows that the family of the hypersimplices $\mathcal H_{r,d,i}$ generates a (hyper) triangulation of $r\Delta_{d,i}$. The dual graph of this (hyper) triangulation has a simple description.

Recall that the vertices of the dual graph of a triangulation are the maximal faces of the triangulation and two maximal faces are connected in the dual graph if their intersection is a maximal face of both of them.

The maximal faces of the hypersimplex $\Delta_{d,j}$ are given with 

\begin{align*}
&\bm e_t+\mathrm{conv}\left\{\bm e_T:T\in\binom{[d+1]\setminus\{t\}}{j-1}\right\}\ \text{and}\\
&\mathrm{conv}\left\{\bm e_T:T\in\binom{[d+1]\setminus\{t\}}{j}\right\},\ t\in[d+1].
\end{align*}

(When $j=1$ there are no maximal faces of the first kind and when $j=d$ there are no maximal faces of the second kind.)

From (\ref{intersections}) we can see that two hypersimplices $\bm u+\Delta_{d,j_1}$ and $\bm v+\Delta_{d,j_2}$ from $\mathcal H_{r,d,i}$ are connected in the dual graph of the triangulation if and only if $\bm u$ and $\bm v$ differ by 1 at exactly one position, i.e. $\bm v=\bm u+\bm e_t$ or $\bm u=\bm v+\bm e_t$ for some $t\in[d+1]$. Thus, the dual graph of this triangulation is a subgraph of the lattice $\mathbb Z^{d+1}$ induced by $\mathcal C_{r-1,d+1,ir-j}$.

\end{document}